\documentclass[a4paper, UKenglish]{amsart}

\usepackage{amsmath}
\usepackage{amsthm}
\usepackage{amssymb}
\usepackage[foot]{amsaddr}
\usepackage{mathtools}
\usepackage{stmaryrd}
\usepackage{hyperref}
\usepackage{cleveref}
\usepackage{url}
\usepackage{tikz-cd}
\usepackage{tikz}
\usepackage{tkz-euclide}

\usepackage{xcolor}

\usepackage{caption}

\DeclareMathOperator{\RRR}{\mathbb{R}}

\DeclareMathOperator{\II}{\mathcal{I}}
\DeclareMathOperator{\JJ}{\mathcal{J}}

\DeclareMathOperator{\FF}{\mathcal{F}}

\DeclareMathOperator{\RR}{\mathcal{R}}
\DeclareMathOperator{\XX}{\mathcal{X}}

\DeclareMathOperator{\PPP}{\mathcal{P}}

\renewcommand*{\phi}{\varphi}

\DeclareMathOperator{\Mod}{Mod}

\DeclareMathOperator{\fp}{fp}

\DeclareMathOperator{\Img}{Im}

\DeclareMathOperator{\Ker}{Ker}

\DeclareMathOperator{\id}{id}

\DeclareMathOperator{\Hom}{Hom}

\newcommand{\frincipal}[1]{#1 ^+}

\newcommand{\Lan}[1]{\mathrm{Lan}_{#1}}

\newcommand{\KerDir}[2]{\Ker {#1}^{\overrightarrow{#2}}}
\newcommand{\ImDir}[2]{\Img {#1}^{\overrightarrow{#2}}}

\theoremstyle{plain}
\newtheorem{theorem}{Theorem}
\newtheorem{prop}[theorem]{Proposition}
\newtheorem{lem}[theorem]{Lemma}

\theoremstyle{definition}
\newtheorem{defi}[theorem]{Definition}
\newtheorem{pbm}[theorem]{Problem}

\theoremstyle{remark}
\newtheorem{ex}[theorem]{Example}
\newtheorem{rk}[theorem]{Remark}

\Crefname{prop}{Proposition}{Propositions}
\Crefname{lem}{Lemma}{Lemmas}
\Crefname{cor}{Corollary}{Corollaries}
\Crefname{defi}{Definition}{Definitions}
\Crefname{ex}{Example}{Examples}
\Crefname{rk}{Remark}{Remarks}
\Crefname{pbm}{Problem}{Problems}

\usepackage{comment}

\newcommand{\field}{\mathbf{k}}
\newcommand{\Vect}{\mathsf{Vect}_\field}

\DeclareMathOperator{\EE}{\mathcal{E}}

\colorlet{colorsteve}{blue}
\colorlet{colorstevetodo}{red}

\colorlet{colorvadim}{magenta}

\colorlet{colorjp}{orange}

\newcommand{\cplx}[1][M]{K_C(#1)}

\title[Block-decomposability of multiparameter persistence modules]{Local characterization of block-decomposability for multiparameter persistence modules}
\author{Vadim Lebovici}
\email{lebovici@math.univ-paris13.fr}
\address{LAGA, Université Sorbonne Paris Nord}

\author{Jan-Paul Lerch}
\email{lerch@math.uni-bielefeld.de}
\address{Faculty of Mathematics, Bielefeld University}

\author{Steve Oudot}
\email{steve.oudot@inria.fr}
\address{Inria Saclay and École polytechnique}

\begin{document}
\begin{abstract}
  Local conditions for the direct summands of a persistence module to belong to a certain class of indecomposables have been proposed in the 2-parameter setting, notably for the class of indecomposables called block modules, which plays a prominent role in levelset persistence. Here we generalize the local condition for decomposability into block modules to the $n$-parameter setting, and prove a corresponding structure theorem. Our result holds in the generality of pointwise finite-dimensional modules over finite products of arbitrary totally ordered sets. Our
  proof extends the one by Botnan and Crawley-Boevey from $2$ to $n$ parameters, which requires some crucial adaptations at places where their proof is fundamentally tied to the 2-parameter setting.
\end{abstract}
\maketitle

\section{Introduction}\label{sec:introduction}
\subsection{Context}
It has been known since the beginning of the development of multiparameter persistence  that there is no complete discrete descriptor for multiparameter persistence modules~\cite{Carlsson2009}. As a consequence, much of the effort has been put in the design of incomplete discrete descriptors that can capture as much of the modules' structure as possible.  A fruitful idea in this direction has been to use homological approximations, in a nutshell: given a persistence module~$M\colon \RRR^n\to\Vect$, take a resolution $P_\bullet$ of $M$ by projectives, then build a descriptor from the supports of the direct summands of the $P_i$'s.
When $M$ is finitely presentable~($\fp$), the $P_i$'s can be taken to be free modules, and Hilbert's Syzygy theorem ensures that $M$ admits finite such resolutions, meaning that there are only finitely many non-zero $P_i$'s in the exact sequence, and that each one of these terms decomposes as a finite direct sum of interval modules supported on principal upsets, that is:
\[ P_i\simeq \bigoplus_{x \in \xi_i} \field_{x^+}, \]
where $\xi_i$ is a multiset of points in $\RRR^n$, called the $i$-th {\em multigraded Betti number} of~$M$, and where $x^+$ denotes the principal upset  $\{y\in\RRR^n \mid y \geq x \}$. These intervals can then be gathered into a descriptor called the {\em signed barcode} of~$M$, whose positive part is composed of those upsets coming from even degrees in the resolution, and whose negative part  is composed of the upsets coming from odd degrees, with multiplicity. Uniqueness of this barcode is guaranteed insofar as minimal free resolutions are considered, since in this case the $P_i$'s are unique up to isomorphism. Signed barcodes coming from minimal resolutions were introduced in~\cite{botnan2021signed} as a generalization of the usual (unsigned) barcodes from 1-parameter persistence, and they were shown to be amenable to interpretation in similar ways. They were then proven in~\cite{oudot2023stability} to be also stable under a signed version of the usual bottleneck distance between barcodes.  

The problem with free resolutions is that they offer little variety as to the shapes of the intervals involved in the signed barcodes, which have to be principal upsets. This is why the use of larger collections of intervals has since been proposed, including: the whole collection of upsets~\cite{miller2020homological}, the {\em hooks} or the {\em rectangles}~\cite{botnan2021signed}, the collection of {\em single-source intervals}~\cite{blanchette2021homological}, or even the whole collection of intervals~\cite{asashiba2023approximation}. The general framework was set up in~\cite{blanchette2021homological} using relative homological algebra. Given a collection~$\II$ of intervals containing the principal upsets, the idea is to declare the interval modules supported on elements in~$\II$ as our new class~$\XX$ of indecomposable projectives. This amounts to restricting the focus to those short exact sequences on which the $\Hom$-functor $\Hom(X,-)$ is exact for every $X\in\XX$. Under some conditions, the collection~$\EE_{\XX}$ of such sequences defines the structure of an exact category on persistence modules, which implies that standard tools from homological algebra can be used. In particular, if a module~$M$ admits finite resolutions by projectives relative to~$\EE_{\XX}$ (also called {\em $\XX$-resolutions}), then the minimal such resolutions give rise to a unique signed barcode, in which each {\em bar} is the support of some element in~$\XX$  and therefore an interval from the collection~$\II$.

A central follow-up question is how to compute finite $\XX$-resolutions, or more precisely the supports of the summands appearing in their various non-zero terms. To our knowledge this is a widely open question, with some interesting recent contributions using Koszul complexes~\cite{asashiba2023relative,chacholski2022effective} but still a long way to go towards scalable methods. In this context, we turn to the following decision version of the problem:
\begin{pbm}\label{pbm:decision_resolution}
  Given a fixed collection~$\II$ of intervals in $\RRR^n$, its corresponding class~$\XX$ of indecomposable projectives, and a fixed natural number $k$, can we efficiently  decide whether or not a given $\fp$ persistence module~$M$ admits an $\XX$-resolution of length at most~$k$? Here~$M$ can be given either by some finite free presentation, or by some simplicial multifiltration.
\end{pbm}
In this paper we focus on the case $k=0$, which asks to decide whether or not $M$ is projective relative to~$\EE_{\XX}$, or equivalently, whether or not $M$ is {\em $\XX$-decomposable}, meaning that it decomposes as a direct sum of elements in~$\XX$. This problem has received some attention for~$n=2$, where it was addressed for $\II$ being either the collection of {\em rectangles} (i.e., products of 2 intervals of~$\RRR$) or the collection of {\em blocks} (i.e., rectangles that are either infinite horizontal or vertical bands, or codirected upsets or directed downsets):
\begin{theorem}[\cite{Botnan2018a,botnan2022rectangle,botnan2023local,Cochoy2020}]\label{thm:2d-block-rect-decomp}
A $\fp$ persistence module $M\colon \RRR^2 \to \Vect$ is block-decomposable (resp. rectangle-decomposable) if, and only if, its restriction to the commutative square $\{r, r'\}\times\{s,s'\}$ is block-decomposable (resp. rectangle-decom\-posable) for every $r<r'\in \RRR$ and $s<s'\in \RRR$.
The result still holds if $M$ is only pfd (pointwise finite-dimensional), or if $\RRR^2$ is replaced by the cartesian product of any two totally ordered sets. 
\end{theorem}
Results such as this one are called {\em local characterizations} of $\XX$-decomposability, because they decompose the problem of determining $\XX$-decomposability for~$M$ into multiple instances over  small posets. The benefit is twofold. First, for the mathematician, it makes it easier to show that the modules arising from certain constructions satisfy unique properties. For instance, the fact that the 2-parameter modules arising from levelset persistence on the real line are block-decomposable is a direct consequence of Theorem~\ref{thm:2d-block-rect-decomp} and the Mayer-Vietoris sequence~\cite{Botnan2018,CdM09,carlsson2019parametrized}. Second, for the computer scientist, local characterizations may eventually lead to efficient algorithms for checking $\XX$-decomposability and for decomposing $\XX$-decomposable modules, as it happens for instance in the case of rectangles~\cite{botnan2022rectangle}.

\subsection{Our contributions}

Here we generalize  Theorem~\ref{thm:2d-block-rect-decomp} to products of $n\geq 3$  totally ordered sets,
focusing on the case where $\II$ is the collection of $n$-dimensional blocks, defined inductively as follows:
\begin{defi}\label{def:block}
  An interval $I$ of a product of $n\geq 3$ totally ordered sets $T_1 \times \ldots \times T_n$ is an \emph{$n$-block} if it is the Cartesian product of either:
  \begin{itemize}
  \item non-empty upward-closed proper subsets of each axis \emph{(birth block)}, or
  \item non-empty downward-closed proper subsets in each axis \emph{(death block)}, or
  \item an $(n-1)$-block in the product of $n-1$ axes $T_{i_1}\times \ldots \times T_{i_{n-1}}$, $1\leq i_1<i_2<\cdots <i_{n-1}\leq n$, and the entire remaining axis $i_n\notin \{i_1, i_2, \cdots, i_{n-1}\}$ (\emph{induced block}). 
  \end{itemize}
The corresponding interval module supported on~$I$ (see \Cref{sec:background}) is called a {\em block module}, and a persistence module isomorphic to a direct sum of block modules is called \emph{block-decomposable}.
\end{defi}

Our main result generalizes Theorem~\ref{thm:2d-block-rect-decomp} as follows: 
\begin{theorem}\label{thm:blockDecomp}
A pfd persistence module $M \colon T_1\times \ldots \times T_n \to \Vect$  over the Cartesian product of $n$ totally ordered sets is block-decomposable if, and only if, its restrictions to all commutative cubes $\{t_1, t_1'\}\times \ldots \times \{t_n, t_n'\}$ with~$t_i<t_i'\in T_i$ are block-decomposable.  
\end{theorem}

Since restriction commutes with pointwise finite direct sums, and block modules restrict to block modules on cubes, the ``only if'' part of the result is immediate. Our proof of Theorem~\ref{thm:blockDecomp} therefore focuses on the ``if'' part. It is phrased using an alternative algebraic formulation for the local block-decomposability of~$M$, based on the concept of {\em $n$-middle exactness}, which, in short, says that, in any commutative cube $C$ in the poset, the Koszul complex of $M|_C$ at the maximal vertex of~$C$ has trivial homology in all but extremal degrees; see \Cref{def:Koszul,def:mex} and \Cref{thm:blockDecomp-exactness}. This is a natural extension to the $n$-parameter setting of the concept of {\em 2-middle exactness} that is used as algebraic formulation for local block-decomposability in the 2-parameter setting.
In fact, $k$-middle exactness in each cube for all $2\leq k \leq n$ are required for $M$ to be locally block-decomposable in the $n$-parameter setting (see Examples~\ref{ex:2-mid_not_3-mid} and~\ref{ex:3-mid_not_2-mid}), and reciprocally, locally block-decomposable modules are $k$-middle exact for all $2\leq k \leq n$.

Our proof of the ``if'' part of \Cref{thm:blockDecomp}  extends the one in \cite{Botnan2018a} with crucial adaptations at places where the original proof is fundamentally tied to the 2-parameter setting. In \cite{Botnan2018a}, the authors proceed by case distinction between those middle-exact persistence modules that are {\em $2$-exact}, i.e., whose Koszul complexes have trivial homology in all degrees (including the extremal ones), and those middle-exact persistence modules that are not $2$-exact. Here, we divide our case analysis between middle-exact modules that are not $n$-exact (\Cref{sec:not-exact}), middle-exact modules that are $k$-exact only for $k$ greater than or equal to some $l \geq 3$ (\Cref{sec:properlyBoundedExactness}), and middle-exact modules that are $k$-exact for all $k$ (\Cref{sec:2-exact}). Our most significant adaptations reside in the proof of the third case, for which we consider the  special case where the poset $T_1\times\ldots \times T_n$ is an $n$-dimensional cube as an intermediate step.

We should point out that, in contrast to what happens when~$n=2$, blocks are no longer tied to levelset persistence when $n>2$, including when $n$~is even. Yet, the case of blocks still has an interest in its own right. First, in the $\fp$ setting over~$\RRR^n$: the only block modules that can occur are the ones supported on principal upsets, therefore  our result becomes a local characterization of the free modules. Second, in the $\fp$ setting over a finite $n$-dimensional grid: since every death block restricts to a death block over at least one cube, our result induces a local characterization of the modules that decompose into block summands whose supports are not death blocks. These form an intermediate collection of interval modules between the free modules and the hook-decomposable modules. Therefore, in terms of fineness as an invariant, the Euler characteristic of relative resolutions by non-death block modules (which collapses the resolutions by taking the alternating sum of their terms of all degrees, in the associated Grothendieck group) lies in-between the Euler characteristic of free resolutions (known to be equivalent to the dimension vector or Hilbert function) and the Euler characteristic of relative resolutions by hook-decomposable modules (known to be equivalent to the rank invariant).

\subsection{Related work}
To our knowledge there are two alternative approaches for solving Problem~\ref{pbm:decision_resolution} in the case $k=0$. The first one consists in decomposing the module~$M$ explicitly, then checking the summands one by one. Already when $n=2$, the best known algorithm~\cite{dey2022generalized} for decomposing~$M$ into its indecomposable summands has a time complexity bound in $O(N^{2\omega+1})$, where $N$ is the number of simplices in the bifiltration from which $M$ originates, and where $\omega\geq 2$  is the exponent for matrix multiplication. By contrast, the method derived from Theorem~\ref{thm:2d-block-rect-decomp} to check rectangle-decomposability on the same input runs in $O(N^{2+\omega})$~time~\cite{botnan2022rectangle}. The second approach consists in computing the multigraded Betti number in degree~$1$ relative to~$\EE_{\XX}$, and to check whether it is the empty set. In the special case where $n=2$ and $\XX$ is the usual class of projectives, this can be done in time $O(N^3)$~\cite{lesnick2022computing,kerber2021fast}. In more general cases, the method of~\cite{asashiba2023relative,chacholski2022effective} based on Koszul complexes can be used for this purpose, but it requires to compute the $\Hom$-spaces from all the elements of~$\XX$ to~$M$, a task that already takes at least $\Omega(N^5)$ time (and probably more) for the rectangles on a 2-dimensional $N\times N$~grid.  This approach though has the advantage of making it possible to solve Problem~\ref{pbm:decision_resolution} for larger values of~$k$, by computing the multigraded Betti number in degree~$k+1$ relative to~$\EE_{\XX}$. Still using Koszul complexes, it is possible to tackle \Cref{pbm:decision_resolution} directly by computing the length of the resolution (or at least some upper bound) without computing the resolution itself. This is the task undertaken in~\cite{chacholski2024realisationsposetstameness} in the case of the usual class of projectives, which opens some perspectives for potential generalizations to other classes including the block-decomposable modules.

%
%

\subsection*{Acknowledgement}
The authors would like to thank Fran\c cois Petit for initial discussions that sparked this project. The first named author was funded in part by EPSRC EP/R018472/1. For the purpose of Open Access, the author has applied a CC BY public copyright licence to any Author Accepted Manuscript (AAM) version arising from this submission. The second author acknowledges that he has been supported by the Alexander von Humboldt Foundation in the framework of an Alexander von Humboldt Professorship endowed by the German Federal Ministry of Education and Research and that he is currently funded by the Deutsche Forschungsgemeinschaft (DFG, German Research Foundation) – Project-ID 491392403 – TRR 358.
He would also like to thank his PhD advisor Henning Krause, as a version of this proof was one part of his PhD project, William Crawley--Boevey and Benedikt Fluhr.

\section{Background} \label{sec:background}
In this section, we formally introduce our notations, as well as the existing results on block-decomposable 2-parameter persistence modules.

\subsection{Persistence modules} Throughout the article, let  $(P,\leq)$ be a poset, denoted simply by $P$, and let $\field$ be a field. A \emph{persistence module} over $P$ is a collection of vector spaces $\{M(p)\}_{p\in P}$ together with linear maps $M(p\leq q) \colon M(p)\to M(q)$ for any comparable pair $p\leq q$ in $P$. In the language of category theory, a poset $P$ canonically defines a category with elements of $P$ as objects and comparable pairs $p\leq q$ in $P$ as morphisms between them. In this setting, a persistence module is a functor $M\colon P\rightarrow \Vect$ from $P$ to the category of $\field$-vector spaces. \emph{Morphisms} between two persistence modules $M$ and $N$ over $P$ are given by natural transformations between them, that is, by collections of linear maps $\phi_p \colon M(p) \to N(p)$ subjected to the commutativity relations $N(p\leq q) \circ \phi_p = \phi_q \circ M(p\leq q)$ for all $p\leq q$ in~$P$. Similarly, monomorphisms, epimorphisms and isomorphisms, as well as direct sums, kernels, cokernels and images, are defined pointwise at each $p \in P$. The persistence modules over~$P$ thus form an abelian category, which we denote by~$\Mod \field P$. If $M\in\Mod \field P$ and $m\in M(p)$ for some $p\in P$, we denote~$|m|=p$. Any poset morphism $f \colon P \to Q$ induces a functor called \emph{pullback} $f^*\colon \Mod \field Q \to \Mod \field P$ defined by $f^* M (-) = M (f(-))$.

A persistence module $N\neq 0$ over $P$ is called a \emph{summand} of $M$ if we have $M \cong N\oplus N'$ for some third persistence module $N'$ over $P$. If both $N$ and $N'$ are non-zero, then $M$ is called {\em decomposable}. Modules that are not decomposable are called {\em indecomposable}. It is a known fact that every module~$M$ that is {\em pointwise finite-dimensional} (pfd), meaning that $\dim M(p)<\infty$ for all $p\in P$, decomposes essentially uniquely as a (pointwise finite) direct sum of indecomposable summands~\cite{Botnan2018a}. The {\em support} of a persistence module $M$ is the set of points $p\in P$ such that $M(p)$ is non-zero. A \emph{submodule} $N$ of $M$ is a persistence module together with a monomorphism $N\hookrightarrow M$. 

\subsection{Kan extensions and restrictions}

A classical way to extend persistence modules from smaller posets to larger posets is via Kan extensions. We recall this notion in the specific case of persistence modules, as in~\cite{Botnan2018}. Let $Q$ be a full subposet of $P$ (i.e., a subset of~$P$ on which the order relation is the restriction of the one on~$P$), and let $M$ be a persistence module over $Q$. The \emph{left Kan extension} of $M$ along the inclusion $\iota:Q\hookrightarrow P$ is the persistence module $\Lan{\iota}(M)$ over $P$ defined pointwise by the colimit formula:
\begin{equation*}
\forall p\in P,\quad  \Lan{\iota}(M)(p) = \varinjlim M|_{\iota(Q)\cap p^-}, \quad \text{where } p^- = \{p'\in P\colon p' \leq p\}. 
\end{equation*} 
The morphisms $\Lan{\iota}(M)(p\leq p')$ are defined from the functoriality of colimits.
It is known that $\Lan{\iota}\colon \Mod \field Q\to \Mod \field P$ is left adjoint to the pullback by the inclusion~$\iota^*\colon \Mod \field P \to \Mod \field Q$ called {\em restriction functor}. This implies in particular that:
\begin{equation}\label{eq:Lan_restr_adj}
\Hom_{\Mod \field Q} \left( \iota^* M ,\iota^* M  \right) \cong  \Hom_{\Mod \field P} \left( \Lan {\iota} \left( \iota^* M \right) , M \right).
\end{equation}
This is even an isomorphism of $\field$-algebras, as the involved morphisms are natural. In the following we sometimes use the alternative notation $M|_Q$ for the restriction of~$M$ to~$Q$, especially when the inclusion map~$\iota$ is not explicit.

\subsection{Intervals and interval modules}
Recall that an \emph{interval} of $P$ is a subset $I \subseteq P$ which is \emph{convex} (i.e., if $x \leq y \leq z$ with~$x, z \in I$, then $y \in I$) and \emph{connected} (i.e., for every two points $x, y \in I$, there exists a finite zig-zag~$x = z_0 \leq z_1 \geq \ldots \leq z_{n-1} \geq z_n = y$ of relations with $z_i\in I$ for all $i=0,\ldots,n$). To each interval~$I\subset P$ corresponds a unique \emph{interval module} with support $I$, denoted by~$\field_I$ and defined as:
\begin{align*}
  &\field_I(p) = 
  \begin{cases}
    \field  & \mbox{if }p\in I,\\
    0       & \mbox{otherwise,}
  \end{cases}
  &\field_I(p\leq q) = 
  \begin{cases}
    \id_\field  & \mbox{if }p,q\in I,\\
    0       & \mbox{otherwise.}
  \end{cases}
\end{align*}
Interval modules are indecomposable because their endomorphism ring is isomorphic to~$\field$ hence local. A persistence module over $P$ that decomposes as a direct sum of interval modules is called \emph{interval-decomposable}.

\subsection{Block-decomposability and $2$-middle exactness}

Let $P=R\times S$ be a product of totally ordered sets endowed with the product order. An important class of intervals of $R \times S$ are the so-called {\em blocks}. A block is the product of either:\begin{itemize}
  \item two non-empty upward-closed proper subsets of each axis, in which case it is called a \emph{birth block}, or
\item two non-empty downward-closed proper subsets of each axis, in which case it is called a \emph{death block}, or
\item the product of an interval on one axis with the entire other axis, in which case it is called a {\em band}.
\end{itemize}
An interval module with support a block is called a \emph{block module}, and a persistence module over $R\times S$ is called \emph{block-decomposable} if it is interval-decomposable and the supports of its summands are blocks.

As mentioned in the introduction, there exists a block-decomposition result for persistence modules indexed over $R\times S$. The proof of this result relies on a local algebraic characterization of block-decomposability, called ``middle-exactness'', which we recall here. Let $M$ be a persistence module over $R \times S$. An \emph{axis-aligned commutative square} in $R\times S$, or simply a \emph{square}, is a subset $Q = \{r<r'\}\times\{s<s'\}\subseteq R \times S$. For the sake of notational clarity, we identify the poset $Q$ with the power set of $\{1,2\}$ endowed with inclusion:
\begin{equation*}
  \begin{tikzcd}
    \lbrace 2 \rbrace \rar \arrow[from=d] & \lbrace 1,2 \rbrace \arrow[from=d,] \\
    \emptyset \rar & \lbrace 1 \rbrace
  \end{tikzcd}
\end{equation*}
and we suggestively denote the restriction $M|_Q$ of $M$ to $Q$ by the following notation:
\begin{equation*}
  \begin{tikzcd} 
  M_{2} \rar["d_1^{12}"] \arrow[from=d, "d^2"] & M_{12} \arrow[from=d, "d_2^{12}"] \\
  M_{\emptyset} \rar["d^1"] & M_1
  \end{tikzcd}
\end{equation*}
The persistence module $M$ over $P$ is called \emph{$2$-middle exact} if, for every square $Q\subset R \times S$,  the complex~$K_Q(M)$ defined as
\begin{equation} \label{eq:2d-Koszul}
0\longrightarrow  M_\emptyset \xrightarrow{\left[ \begin{smallmatrix} d^1 \\ d^2\end{smallmatrix} \right]} M_1 \oplus M_2 \xrightarrow{ \left[\begin{smallmatrix} d_2^{12}&\ - d_1^{12} \end{smallmatrix} \right]} M_{12} \longrightarrow 0,
\end{equation}
 has trivial homology in degree~1, i.e. $H_1\left(K_Q(M)\right) = 0$, where the term $M_{12}$ is in degree~0.
 This is equivalent to asking that the sequence is exact in the middle, meaning that the image of the map $M_\emptyset \to M_1\oplus M_2$ is equal to the kernel of the map $M_1\oplus M_2 \to M_{12}$.
  The module~$M$ is called \emph{$2$-left exact} (resp. \emph{$2$-right exact}, \emph{$2$-exact}) if it is $2$-middle exact and in addition the complex~$K_Q(M)$ has trivial homology in degree~2 (resp. in degree~0, in both degrees) for every square $Q \subset R \times S$. When $M$ is $2$-right exact, the restriction~$M|_Q$ is called a {\em pushout square}, or simply a {\em pushout}, because it then forms the diagram of the colimit of the span $M_2\leftarrow M_\emptyset \rightarrow M_1$.

\begin{ex}
 The block module $\field_B$ for a block $B\subseteq R \times S$ is 2-middle exact. One can also easily check that it is 2-left exact if, and only if, $B$ is not a death block, and that it is 2-right exact if, and only if, $B$ is not a birth block.
\end{ex}

It can be checked that a persistence module $M$ over a square $Q\subset R\times S$ is block-decomposable if, and only if, it is 2-middle exact. As a consequence, \Cref{thm:2d-block-rect-decomp} for block-decomposability is equivalent to the following local algebraic characterization:
\begin{theorem}[\cite{Botnan2018a,Cochoy2020}] 
\label{thm:BCB}
  Let $M$ be a pointwise finite-dimensional persistence module over~$P=R\times S$. Then, $M$ is block-decomposable if, and only if, it is 2-middle exact.
\end{theorem}
  
\section{Higher-dimensional middle-exactness and blocks}
\label{sec:3d-middle-exactness}

Throughout the article, let $n\geq 2$ be an integer and let $P = T_1\times \ldots \times T_n$ be the product of $n$ totally ordered sets, equipped with the product order.
For $1\leq k\leq n$, a \emph{$k$-slice of $P$} is a subset of $P$ defined by fixing exactly $n-k$ components. An \emph{axis-aligned commutative $k$-cube} in $P$, or simply a \emph{$k$-cube} is a subset $C = \{t_1, t'_1\}\times \ldots \times\{t_n, t'_n\}\subseteq P$ where~$t_i\leq t'_i$ are in $T_i$ and where exactly $k$ of these inequalities are strict. In particular, a $k$-cube is included in a $k$-slice but is not contained in any $(k-1)$-slice. For the sake of notational clarity, we identify the poset $C$ with the power set of $[k]=\{1,\ldots,k\}$ endowed with inclusion, and we write $\delta_{x}^A$ for the arrow $A\setminus\{x\}\to A$ when $A\setminus\{x\}\neq\emptyset$. For instance, when~$k=3$, we have:
\begin{equation}\label{eq:cube}
  \begin{tikzcd}[row sep=scriptsize, column sep=scriptsize] 
      & \lbrace 2,3 \rbrace \arrow[rr, "\delta_1^{123}"] \arrow[from=dd, "\delta_2^{23}" near start] & & \lbrace 1,2,3 \rbrace \\
      \lbrace 2 \rbrace \arrow[ur, "\delta_3^{23}"] \arrow[rr, crossing over, "\delta_1^{12}" near end] & & \lbrace 1,2 \rbrace \arrow[ur, "\delta_3^{123}"] \\
      & \lbrace 3 \rbrace \arrow[rr, "\delta_1^{13}" near start]  & & \lbrace 1,3 \rbrace \arrow[uu, "\delta_2^{123}"] \\
      \emptyset \arrow[ur, "\delta^3"] \arrow[rr, "\delta^1"] \arrow[uu, "\delta^2"] & & \lbrace 1 \rbrace \arrow[ur, "\delta_3^{13}"] \arrow[uu, crossing over, " \delta_2^{12}" near start]
    \end{tikzcd}
\end{equation}
We then suggestively denote the structure spaces and morphisms of the restriction~$M|_C$ of a persistence module~$M\colon  P \to \Vect$ to~$C$ by~$d_x^A\colon M_{A\setminus\{x\}}\to M_A$, and simply $d^i\colon M_\emptyset \to M_i$ instead of $d_i^{\{i\}}$. For instance, when~$k=3$, we denote:
\begin{equation}\label{eq:cubeRep}
  \begin{tikzcd}[row sep=scriptsize, column sep=scriptsize]
  & M_{ 23 } \arrow[rr, "d_1^{123}"] \arrow[from=dd, "d_2^{23}" near start] & & M_{ 123 } \\
  M_ 2  \arrow[ur, "d_3^{23}"] \arrow[rr, crossing over, "d_1^{12}" near end] & & M_{ 12 } \arrow[ur, "d_3^{123}"] \\
  & M_ 3  \arrow[rr, "d_1^{13}" near start]  & & M_{ 13 } \arrow[uu, "d_2^{123}"'] \\
  M_{\emptyset} \arrow[ur, "d^3"] \arrow[rr, "d^1"] \arrow[uu, "d^2"] & & M_{1} \arrow[ur, "d_3^{13}"] \arrow[uu, crossing over, " d_2^{12}" near start]
  \end{tikzcd}
\end{equation}

The notion of middle exactness is expressed using Koszul complexes \cite{Miller2005,Eisenbud1995,chacholski2022effective}. We do not recall the standard definition of Koszul complexes, but we provide an equivalent description of $\cplx$ in our setting. 
\begin{defi}\label{def:Koszul}
  Let $C$ be a $k$-cube of $P$ and $M \colon C\to \Vect$ be a persistence module over $C$. We define the \emph{Koszul complex} $\cplx$ by induction on $k$ as follows:
  \begin{enumerate}
    \item If $k=2$, then $\cplx$ is given by~\eqref{eq:2d-Koszul}. Moreover, one can readily check that any morphism of poset $\phi \colon C \to C'$ induces a morphism of chain complexes $\Phi \colon  K_{C}(M) \to K_{C'}(M)$.
    \item If $k\geq 3$, consider the opposite faces of $C = \PPP([k])$ defined by $\FF:=\PPP([k-1])$ and $\RR:=\PPP([k-1])* \lbrace k \rbrace$, where the latter denotes the set of all elements of $C$ which contain $k$.
    The Koszul complexes $K_{\FF}(M)$ and $K_{\RR}(M)$ of $M$ on these respective lower dimensional faces are defined by induction. Moreover, the morphism of poset $\phi \colon \FF\to\RR$ given by $S \mapsto S \cup \lbrace k \rbrace$ induces a morphism of chain complexes $\Phi \colon K_{\FF}(M)\rightarrow K_{\RR}(M)$. The complex~$\cplx$ is defined as the cone of this morphism $\cplx := \textrm{Cone}(\Phi)$.
  \end{enumerate}
\end{defi}
The complex $\cplx$ of our definition is the Koszul complex of $M|_C$ at the maximal element of $C$ following the convention of \cite[Def.~1.26, Ex.~1.27]{Miller2005}. This also shows that the above construction does not depend on the choices of opposite faces $\FF$ and $\RR$.

\begin{ex}\label{ex:3-Koszul}
  Let $C$ be a $3$-cube and $M$ be a persistence module over $C$. Recall the notations of~\eqref{eq:cubeRep}. The complex $\cplx$ is then given by:
  \begin{equation*} 
  \begin{tikzcd}[column sep = normal]
0 \arrow[r] &     M_\emptyset \arrow[r, "{\left[ \begin{smallmatrix}d^1 \\ d^2 \\ d^3 \end{smallmatrix} \right]}"] & M_1\oplus M_2 \oplus M_3 \arrow[r, "A"] & M_{12}\oplus M_{13}\oplus M_{23} \arrow[r, "{\left[ \begin{smallmatrix}  d_3^{123} \\ -d_2^{123} \\ d_1^{123} \end{smallmatrix} \right]^T}"] & M_{123} \arrow[r] & 0,
    \end{tikzcd}
  \end{equation*}
  where the term $M_{123}$ is in degree $0$ and where $A$ is the matrix:
  \begin{equation*}
    A=\left[ \begin{matrix} d_2^{12} & -d_1^{12} & 0 \\ d_3^{13} & 0 & -d_1^{13} \\ 0 & d_3^{23} & -d_2^{23} \end{matrix} \right].
  \end{equation*}
\end{ex}

We can now generalize the notion of middle-exactness to the $n$-parameter setting.

\begin{defi}\label{def:mex}
  A persistence module $M : P\to \Vect$ is \emph{$k$-middle exact} on a $k$-cube~$C$ if the complex $\cplx$ has trivial homology in all degrees $0 < i <k$. The persistence module $M$ is \emph{$k$-left-exact} (resp. \emph{$k$-right-exact}, resp. \emph{$k$-exact}) on $C$ if it is $k$-middle exact on $C$ and in addition the complex $\cplx$ has trivial homology in degree $k$ (resp. in degree 0, resp. in both degrees). A persistence module over~$P$ is \emph{$k$-middle exact} (resp. \emph{$k$-left exact}, \emph{$k$-right exact}, \emph{$k$-exact}) if it is such on every $k$-cube of $P$.
\end{defi}

The following lemma relates the different notions of exactness defined above.
\begin{lem} \label{lem:exactInduction}
  Let $M$ be a pfd persistence module over $P$. If $M$ is $k$-exact for some $2 \leq k \leq n$, then $M$ is $m$-exact for all $k \leq m \leq n$.
\end{lem}
\begin{proof}
  We proceed by induction on $m$. Suppose that $M$ is proven to be $(m-1)$-exact and take a $m$-cube $C$. The faces $\FF$ and $\RR$ of $C$ defined in \Cref{def:Koszul} are $(m-1)$-cubes. Since the Koszul complex $\cplx$ is defined as a mapping cone, we have a triangle in the homotopy category:
  \begin{equation*}
  K_{\FF}(M) \xrightarrow{\Phi} K_{\RR}(M) \rightarrow \cplx \rightarrow K_{\FF}(M) [1].
  \end{equation*}  
  This triangle induces a long exact sequence in homology:
  \begin{multline}\label{eq:les}
  \dots \rightarrow H_{i}\left( K_{\FF} (M)\right) \rightarrow  H_{i}\left( K_{\RR}(M) \right) \rightarrow  \\
  H_{i}\left( K_C(M) \right) \rightarrow H_{i-1}\left( K_{\FF} (M)\right) \rightarrow \dots
  \end{multline}
  and the result follows.
\end{proof}

\begin{rk}
  Assuming that $M$ is $(k-1)$-middle exact, a similar argument to the one used in \Cref{lem:exactInduction} ensures that checking $k$-middle exactness of $M$ can be done by checking exactness of $\cplx$ in degree $1$ and $k-1$ for every $k$-cube $C$. 
\end{rk}

\begin{rk}\label{rk:mexCondition}
  By definition of Koszul complexes, one has $H_{n-1} \left(\cplx\right)=0$ if, and only if, for all $m_i \in M_i$ such that
  \[
  d_i^{\lbrace i, j \rbrace } (m_i) = d_j^{\lbrace i, j \rbrace } (m_j)
  \]
  for all $\lbrace i, j \rbrace \in  C $, there is an element $m_{\emptyset} \in M_{\emptyset}$ with $m_i = d^{\lbrace i \rbrace}_{\emptyset}(m_{\emptyset})$.
  We refer to $m_\emptyset$ as a \emph{lift} of the $m_i$ and to the problem of finding such a lift as the \emph{$n$-lifting problem}.
\end{rk}

The next lemma ensures that block modules are examples of $k$-middle exactness for all $2\leq k \leq n$.
\begin{lem} \label{lem:blocksMex}
  A block module is $k$-middle exact for all $2\leq k \leq n$. 
\end{lem}
\begin{proof}
  The proof proceeds by induction on $n$. The case $n=2$ is an easy exercise, which can also be seen as a consequence of \Cref{thm:BCB}. Suppose $M$ is a block module with support an induced block, which is $k$-middle exact for all $2\leq k\leq n-1$ by induction.
  Then the faces $\FF$ and $\RR$ in the definition of the Koszul complex $\cplx$ over an $n$-cube $C$ can be chosen orthogonal to the axis along which the induced block is extended. In that case, the Koszul complex is the mapping cone of an isomorphism of chain complexes $\Phi:K_{\FF}(M)\to K_{\RR}(M)$, hence has vanishing homology in every degree.
  If $M$ is a death block, then depending on the position of the $n$-cube $C$, we are either back in the previous situation or in a situation where we can choose $\RR$ so that $M_{|\RR}$ is zero. In the latter case, the homology of $\cplx$ coincides with the homology of $K_{\FF}(M)$, which is the Koszul complex of a lower dimensional death block module $M_{|\FF}$, which is middle exact by induction. The result for birth blocks follows by duality.
\end{proof}

Since $k$-middle exactness is closed under taking pfd direct sums of persistence modules, any pfd block-decomposable module is $k$-middle exact for all $2\leq k\leq n$. Our main result states that the converse is also true:
\begin{theorem}\label{thm:blockDecomp-exactness}
  A pfd persistence module $M$ over $P$ is block-decomposable if, and only if, it is $k$-middle exact for all $2\leq k\leq n$.
\end{theorem}

Note that \Cref{thm:blockDecomp-exactness} implies \Cref{thm:blockDecomp}, and vice-versa. Indeed, \Cref{thm:blockDecomp-exactness} ensures that a pfd module that is $k$-middle exact for all $2\leq k \leq n$ is block-decomposable, hence so are all its restrictions to cubes, as block modules restrict to block modules on cubes and as the restriction functor is additive. Reciprocally, a pfd persistence module whose restriction to every cube is block-decomposable is also $k$-middle exact for all $2\leq k \leq n$ by \Cref{lem:blocksMex} and the fact that $k$-middle exactness is closed under taking pfd direct sums.

Finally, note that $k$-middle exactness for all $2\leq k \leq n$ is required to characterize block-decomposable modules in \Cref{thm:blockDecomp-exactness}, as shown by the following two examples.
\begin{ex}\label{ex:2-mid_not_3-mid}
  Consider a $3$-cube $C$ and let $M$ be the persistence module over $C$ defined by:
  \begin{equation*}
    M := \quad
    \begin{tikzcd}[row sep=scriptsize, column sep=scriptsize]
    & 0 \arrow[rr, dashed] \arrow[from=dd, dashed] & & 0 \\
    \field \arrow[ur, dashed] \arrow[rr, dashed, crossing over] & & 0 \arrow[ur, dashed] \\
    & \field \arrow[rr, dashed]  & & 0 \arrow[uu, dashed] \\
    \field^2 \arrow[ur, "{\left[1 \ 1\right]}"] \arrow[rr, "{\left[0\ 1\right]}"] \arrow[uu, "{\left[1\ 0\right]}"] & & \field \arrow[ur, dashed] \arrow[uu, dashed, crossing over]
    \end{tikzcd}
  \end{equation*}
  The persistence module $M$ is indecomposable, as its endomorphism ring is isomorphic to $\field$. Obviously it is not a block module, and it is not $3$-middle exact either. Nevertheless, it is $2$-middle exact on each face of the cube.
\end{ex}

\begin{ex}\label{ex:3-mid_not_2-mid}
Consider again a $3$-cube $C$ identified with the power set of $\{1,2,3\}$ and consider the interval module $\field_I$ with support $I=\left\lbrace \lbrace 2 \rbrace, \lbrace 3 \rbrace, \lbrace 1,3 \rbrace, \lbrace 2,3 \rbrace \right\rbrace$:
\begin{equation*}
  \field_I := \quad
  \begin{tikzcd}[row sep=scriptsize, column sep=scriptsize]
  & \field \arrow[rr, dashed, "0"] \arrow[from=dd, "d_3^{23}" near start] & & 0 \\
  \field \arrow[ur, "d_3^{23}"] \arrow[rr, crossing over,dashed, "0" near end] & &0 \arrow[ur, dashed,"0"] \\
  & \field \arrow[rr, "d_1^{13}" near start]  & & \field \arrow[uu, dashed,"0"] \\
  0 \arrow[ur, dashed,"0"] \arrow[rr,dashed, "0"] \arrow[uu, dashed, "0"] & & 0 \arrow[ur,dashed, "0"] \arrow[uu, crossing over,dashed, "0" near start]
  \end{tikzcd}
\end{equation*} 
The interval $I$ is not a block and $\field_I$ is not $2$-middle exact, as the complex~\eqref{eq:2d-Koszul} associated to the front face of the cube is isomorphic to $0 \rightarrow 0 \rightarrow \field \rightarrow 0 \rightarrow 0$. Nevertheless, the interval module $\field_I$ is $3$-middle exact.
\end{ex}

\subsection*{Outline of the proof.}
This is an extension of the proof in \cite{Botnan2018a} with crucial adaptations at places where the original proof is tied to the 2-parameter setting. In \cite{Botnan2018a}, the authors proceed by case distinction between  those middle-exact persistence modules that are $2$-exact and those that are not. Here we divide the proof into three cases: first, modules that are middle-exact in all degrees but not (properly) exact in any degree (\Cref{sec:not-exact}); second, modules that are middle-exact in all degrees and (properly) exact only in degrees greater than some $l \geq 3$ (\Cref{sec:properlyBoundedExactness}); third, modules that are (properly) exact in every degree (\Cref{sec:2-exact}).

In Case 1, a reduction to the finite case is used (\Cref{lem:discretisation}), to split off birth or death blocks (\Cref{prop:not-n-exact}). This is different from the original proof in \cite{Botnan2018a}, where the problem could be reduced to an alternating poset of type $A_n$. The finite case is dealt with merely using the right-exactness property.

Case $2$ then uses Case $1$ to induce a block decomposition from a block decomposition of the restriction to a lower-dimensional slice (\Cref{pro:Case2}).

The most significant adaptations lie in the proof of Case $3$, for which we again consider the special case where the poset~$P$ is an $n$-cube as an intermediate step. In the $2$-parameter setting, the block-decomposition of a 2-exact persistence module is obtained by using the $1$-parameter interval-decomposition result on the restriction of the module to a well-chosen 1-dimensional subset---namely, a poset of type $A_n$ with outgoing arrows from a central point---then extending it back to the plane. The isomorphism between the original module and the extension of its restriction follows then from the very definition of left-exactness. This methodology simply fails for $n\geq 3$ parameters, as the restriction to a 1-dimensional subset does not contain enough information to recover the original module by extension, even under exactness assumptions.
To mitigate this, we study restrictions to a natural generalization of the 1-dimensional $A_n$-type posets from the 2-parameter setting, with the suggestive name \emph{claw} (\Cref{def:claw}), which have outgoing arrows along each axis from a central point. The main issue is that not all persistence modules indexed over such poset are interval-decomposable when $n\geq 3$. In fact, these posets do not even have finite representation type in general. 
Nonetheless, we show that restrictions of exact persistence modules 
to the claw are interval-decomposable, see \Cref{lem:DecompRestr}. Moreover, as in the 2-parameter setting, extending back these restrictions to the entire $n$-cube recovers the original module, therefore yielding a block decomposition result for $2$-exact persistence modules on $n$-cubes. The use of the finite case to prove the general case is structured in a very similar fashion to the $2$-parameter setting \cite{Botnan2018a}.

\section{Preliminaries}
In this section, we expose two lemmas which will be used repeatedly throughout the rest of the article.
First, we recall the following lemma from \cite{Botnan2018a}.
\begin{lem}[\textnormal{\cite[Lemma~2.1]{Botnan2018a}}]
  \label{lem:injective-summand}
  Let $M$ be a pfd persistence module over $P$. 
  If $M$ admits a persistence submodule~$\field_B\hookrightarrow M$ for some death block $B\subseteq P$, then this submodule is a direct summand. 
\end{lem}

Then, we provide an adaptation of \cite[Lemma~5.4]{Botnan2018a} to the $n$-parameter setting. 

\begin{lem}\label{lem:extension}
Let $M$ be a pfd persistence module over $P$ which is $k$-middle exact for all $2\leq k \leq n$. Let $a_j\in T_j$ for some $j\in\{1,\ldots,n\}$ and let $J_i \subset T_i$  for $i\ne j$ be either the full axis~$T_i$ or an interval admitting an upper bound in $T_i\setminus J_i$.
Then, a monomorphism 
\[
  h:\field_{J_1\times \ldots  \times J_{j-1}\times \{a_j\} \times J_{j+1}\times \ldots \times J_n} \hookrightarrow M_{|T_1\times \ldots \times T_{j-1}\times \{a_j\} \times T_{j+1}\times \ldots \times T_n}.
\]
lifts to a monomorphism
\[
  h':\field_{J_1\times \ldots \times J_{j-1}\times a_j^- \times J_{j+1}\times \ldots \times J_n} \hookrightarrow M_{|T_1\times \ldots \times T_{j-1}\times a_j^- \times T_{j+1}\times \ldots \times T_n},
\]
where $a_j^- = \{t\in T_j \colon t \leq a_j\}$.
\end{lem}

\begin{proof}
  Assume that all $J_i$ are intervals admitting upper bounds in $T_i\setminus J_i$ for $i\ne j$, the other cases being proven similarly. Let $\varepsilon_i \in T_i \setminus J_i$ be upper bounds for $J_i$ for~$i\ne j$ and denote $\varepsilon = \left(\varepsilon_i\right)_{i \neq j}$. Let us denote: 
  \begin{align*}
    P_j &= T_1\times \ldots \times T_{j-1}\times \{a_j\} \times T_{j+1}\times \ldots \times T_n, \\
    P_j^- &=T_1\times \ldots \times T_{j-1}\times a_j^- \times T_{j+1}\times \ldots \times T_n,
  \end{align*}
and the canonical projection by $\pi : P_j^- \to P_j$.
We define a map $\alpha^{\varepsilon}_j : P \to P$ defined for each $p=(p_1,\ldots,p_n)\in P$ by $\alpha^{\varepsilon}_j(p) = (p'_1,\ldots, p'_n)$ where:
\[
  p_i' = 
  \begin{cases}
    p_j & \text{if } i = j, \\
    \varepsilon_i & \text{otherwise.} \\
  \end{cases}
\]
Then we define a persistence submodule $E^{\varepsilon}$ of $M_{|P_j^-}$ for any $p\in P_j^-$ by:
\[
  E^{\varepsilon}(p) := M(p\leq \pi(p))^{-1} \big( \Img h_{\pi(p)} \big) \cap \Ker M(p\leq \alpha_j^{\varepsilon}(p)).
\]
Clearly $E^{\varepsilon}(p) \neq 0 $ for all $p \in P_j$ since then $\pi(p)=p$.
Moreover, it is also non-zero for $p \in P_j^-\setminus P_j$. To prove this, one has to find a preimage $m_\emptyset$ of a non-zero element $m_1\in E^{\varepsilon}(\pi(p))$ along $M(p\leq \pi(p))$ that also lies in $\Ker M(p\leq \alpha_j^{\varepsilon}(p))$. This is a simple consequence of the $n$-middle exactness of $M$ applied to the $n$-cube whose minimum is $p$ and whose maximum is $\alpha_j^{\varepsilon}(\pi(p))$. In the 3-parameter setting, this looks like:
\[ \begin{tikzcd}[row sep=scriptsize, column sep=scriptsize]
  & 0 \arrow[rr, mapsto] \arrow[from=dd] & & 0 \\
  0 \arrow[ur, mapsto]\arrow[rr, crossing over, mapsto] & & 0 \arrow[ur, mapsto] \\
  & 0 \arrow[rr, mapsto]  & & 0 \arrow[uu, mapsto] \\
  \exists m_{\emptyset} \arrow[uu, dashed, mapsto] \arrow[ur, dashed, mapsto]  \arrow[rr, dashed, mapsto]  & & m_1 \arrow[ur, mapsto] \arrow[uu, crossing over, mapsto]
\end{tikzcd}\]

Moreover, the structure maps of $E^\varepsilon$ are epimorphisms. To see this, let $p\leq q \in P_j^- $ and $m \in E^\varepsilon(q)$. We want to find a preimage of $m$ under the map $E^\varepsilon(p) \to E^\varepsilon(q)$. If the cube $C$ with minimum $p$ and maximum~$q$ in~$P$ is a~$k$-cube, then we first lift~$m$ along all $l$-cubes which are adjacent to~$q$ for each $l < k$ in increasing order.
These $l$-lifting problems are solved differently depending on the $l$-cube. If the $l$-cube lies in a slice which contains the $T_j$-axis, then one can solve the $l$-lifting problem by first extending the $l$-cube to a $k$-cube by zero entries and then using the $k$-middle exactness of~$M$. If the cube lies in a slice which is orthogonal to the~$T_j$-axis, then one can solve the $l$-lifting problem by extending the $l$-cube to a $k$-cube with maximum~$\pi(q)$, with entries of the~$k$-cube being either zero if~$M(q\leq \pi(q))(m) = 0$ or the non-zero elements given by the monomorphism~$h$ if~$M(q\leq \pi(q))(m) \in \Img(h_{\pi(q)})\setminus 0$. Finally, we can use the $k$-middle exactness of~$M$ and the previously constructed lifts to find a lift of~$m$ at the minimum~$p$ of~$C$.

Now, for any $p\in P_j^-$ we define:
\[E(p) \coloneqq \bigcap_{\varepsilon} E^{\varepsilon}(p), \]
where the intersection is taken over all $\varepsilon = (\varepsilon_i)_{i\ne j}$ such that $\varepsilon_i > J_i$ for all $i\ne j$.
For any $p\in P_j^-$ there is an $\varepsilon$ such that $E(p) = E^{\varepsilon}(p)$ by pointwise finite dimensionality of $M$. Therefore, the vector spaces $E(p)$ assemble into a persistence submodule of $M_{|P_j^-}$ whose structure maps are epimorphisms.
Thus \cite[Lemma 2.3]{Botnan2018a} ensures that~$E$ has a direct summand $\field_{P_j^-} \hookrightarrow E$, which provides the lift $h'$ of the monomorphism $h$ after multiplication by the appropriate scalar. 
\end{proof}

\section{Case 1: $M$ is not $n$-exact.}\label{sec:not-exact}
In this section, we prove \Cref{thm:blockDecomp-exactness} in the particular case of persistence modules which are not $n$-exact.
\begin{prop}\label{prop:not-n-exact}
  Let $M$ be a pfd persistence module over $P$ which is $k$-middle exact for all $2\leq k\leq n$. If $M$ is not $n$-left exact, then it has a death block summand. Dually, if $M$ is not $n$-right exact, then it has a birth block summand.
\end{prop}
First, we study the middle exactness properties of some kernel submodules introduced in \Cref{sec:study-kernel-submodules}. Then, we extract a death block summand of $M$ within these kernel submodules to prove the above proposition in \Cref{sec:proof-prop-not-n-exact}. 

\subsection{Kernel submodules}\label{sec:study-kernel-submodules}
Let~$M$ be a pfd persistence module over~$P$ which is~$k$-middle exact for all~$2\leq k \leq n$. 
For any $i\in\{1,\ldots,n\}$, we denote by~$\KerDir{M}{i}$ the persistence submodule of~$M$ consisting of all elements which are eventually mapped to zero by the structure maps of~$M$ along the~$T_i$-axis. Throughout the section, we denote~$N = \bigcap_{i=1}^n \KerDir M i$.
\begin{lem} \label{lem:Kersurjective}
The persistence module $N$ has surjective structure maps.
\end{lem}
\begin{proof}
  Every element can be lifted along the axes as in the proof of \Cref{lem:extension}.
\end{proof}

\begin{lem} \label{lem:KernelLifting}
  The persistence module $N$ is $k$-right exact for every $2\leq k\leq n$.
\end{lem}
\begin{proof}
  \Cref{lem:Kersurjective} ensures that for every $k$-cube $C\subseteq P$, we have $H_{0}(\cplx[N])=0$. Therefore, using the long exact sequence~\eqref{eq:les} and following \Cref{rk:mexCondition}, we only have to prove that for every $k$-cube $C\subseteq P$, we have $H_{k-1}(\cplx[N])=0$. By \Cref{rk:mexCondition}, this is the same as to find a solution for the $k$-lifting problem inside $N$. 

  For $k=n$, this follows immediately. We give an illustration of this fact in the case $n=3$. Let $C$ be a $3$-cube inside $P=T_1\times T_2\times T_3$. Since $M$ is $3$-middle exact, we can find an element $m_{\emptyset}$ that fits into the following diagram:
  \[ \begin{tikzcd}[row sep=scriptsize, column sep=scriptsize]
  & 0 \arrow[from=dd] & & \\
  m_2 \arrow[ur, mapsto]\arrow[rr, crossing over, mapsto] & & 0 \\
  & m_3 \arrow[rr, mapsto]  & & 0 \\
  \exists m_{\emptyset} \arrow[uu, dashed, mapsto] \arrow[ur, dashed, mapsto]  \arrow[rr, dashed, mapsto] & & m_1 \arrow[ur, mapsto] \arrow[uu, crossing over, mapsto]
  \end{tikzcd} \]
  where $m_1, m_2$ and $m_3$ are given and assumed to lie in $N$. Yet, this implies that $m_{\emptyset}$ also lies in $N$. 

  If $k<n$, then we can turn this $k$-lifting problem into an $n$-lifting problem by embedding the respective $k$-cube into an $n$-cube in $P$. This cube can be chosen big enough so that all other terms are zero, since the elements of $N$ of the original $k$-cube must be mapped to zero eventually along the remaining $n-k$ axes by definition of $N$. We illustrate this in the case $n=3$ and $k=2$. Consider a $2$-cube~$C$ in~$P$ and the~$2$-lifting problem:
  \begin{equation*}
    \begin{tikzcd}[row sep=scriptsize, column sep=scriptsize]
     \alpha \rar[mapsto]  & \beta \\
     & \gamma \uar[mapsto]
     \end{tikzcd}
  \end{equation*}
  where $\alpha$, $\beta$ and $\gamma$ lie in $N$. By $3$-middle exactness of $M$, one can find a lift $w$ as in the following diagram: 
  \[ \begin{tikzcd}[row sep=scriptsize, column sep=scriptsize]
  & 0 \arrow[from=dd] & & \\
  \alpha \arrow[ur, mapsto]\arrow[rr, crossing over, mapsto] & & \beta \\
  & 0 \arrow[rr, mapsto]  & & 0 \\
  \exists w \arrow[uu, dashed, mapsto] \arrow[ur, dashed, mapsto]  \arrow[rr, dashed, mapsto] & & \gamma \arrow[ur, mapsto] \arrow[uu, crossing over, mapsto]
  \end{tikzcd} \]
 Since $\alpha$ and $\gamma$ lie in $N$, so does $w$.
\end{proof}

\subsection{Proof of \Cref{prop:not-n-exact}}\label{sec:proof-prop-not-n-exact}
We start by showing that, in some cases, one can restrict the study to the finite case (\Cref{lem:discretisation}). Then we prove the result for finite posets (\Cref{lem:finDeathBlock}), and then for general posets. 
But first, let us introduce a special poset, called the~{\em claw}, which will be instrumental both in the proof of \Cref{lem:discretisation} and in the upcoming \Cref{sec:2-exact}. 
\begin{defi}\label{def:claw}
  Let $a=(a_1,\ldots,a_n)\in P$. The \emph{claw} of $P$ at $a$ is the subset~$L:= \bigcup _{i=1}^n L_i \subseteq P$ where the subsets $L_i$ are given by:
  \begin{equation*}
    L_i = \{a_1\}\times \ldots\times \{a_{i-1}\} \times a_i^+ \times \{a_{i+1}\}\times \ldots\times \{a_{n}\}.
  \end{equation*}
  We refer to the subsets $L_i$ as the \emph{arms} of $L$. 
\end{defi}

\begin{lem} \label{lem:discretisation}
  Let $a\in P$ and recall our notation for the principal upset $\frincipal a = \{p\in P : p\geq a\}$. Let $M$ be a pfd persistence module over $a^+$ that is $k$-middle exact for all $2\leq k \leq n$. Assume further that $M$ has only surjective structure maps. Then there exist a finite subposet $P'=T_1'\times \ldots \times T_n' \subseteq a^+$ and a poset morphism~$\pi : a^+ \to P'$ such that $M$ is isomorphic to $\pi^*\left(M_{\mid P'}\right)$.
\end{lem}
\begin{proof}
  Let us write $a=(a_1,\ldots,a_n)$ and denote by $L:= \bigcup _{i=1}^n L_i$ the claw at $a$.
  Since the structure maps of $M$ are surjective and since $M$ is pointwise finite-dimensional, one can partition each arm $L_i$ of $L$ into a disjoint union of finitely many intervals of $T_i$ on which the structure maps of $M_{|L}$ are all isomorphisms. We denote these intervals by $I^k_i \subseteq T_i$, so that $L_i = \coprod_{k} I_i^k$ for all $i=1,\ldots,n$.

 We will show that if~$p \leq q$ are elements of some box~$I_1^{k_1} \times \ldots \times I_n^{k_n}$, then the structure map~$M(p\leq q)$ is an isomorphism.
 We first prove it when~$p$ and~$q$ belong to a~$2$-slice~$I_1^{k_1} \times I_2^{k_2} \times \{t_3 \}\times \ldots \times \{t_n\}$ for~$t_i \in I_i^{k_i}$. This will imply the result for any other similar~$2$-slice by symmetry, hence the result for any~$p$ and~$q$ in~$I_1^{k_1} \times \ldots \times I_n^{k_n}$ by composition.
 So consider the following commutative diagram in~$M$, where $r,s\in a^+$, $u, v \in I_1^{k_1} $ and $m, n \in I_2^{k_2}$:
 \[ \begin{tikzcd}[row sep=scriptsize, column sep=scriptsize]
 M_v \arrow[r] & M_r \arrow[r, , "\beta''"] &  M_q \\
 M_u \arrow[r] \arrow[u, "\alpha"] & M_p \arrow[r, "\beta'"] \arrow[u , "\alpha'"]& M_s \arrow[u, "\alpha''"] \\
 M_a \arrow[r] \arrow[u] & M_m \arrow[r, "\beta"] \arrow[u] & M_n \arrow[u]
 \end{tikzcd} \]
 By definition of $I_1^{k_1}$ and $I_2^{k_2}$, the maps $\alpha$ and $\beta$ are isomorphisms. 
 Since $M$ is $2$-middle exact and has surjective structure maps, it is $2$-right exact, so all squares involved are pushout squares. Therefore, all parallel maps $\alpha', \alpha'', \beta'$ and~$\beta''$ are isomorphisms too, hence so is the structure map $M(p\leq q)$.

 For each $i\in\{1,\ldots,n\}$, let~$T'_i \subseteq T_i$ be a finite subset such that $P'=T'_1\times\ldots\times T'_n$ contains at least one point~$p_{(k_1,\ldots,k_n)}$ in each box~$I_1^{k_1} \times \ldots \times I_n^{k_n}$. Define~$\pi:a^+ \to P'$ on each box~$I_1^{k_1} \times \ldots \times I_n^{k_n}$ by~$p\mapsto p_{(k_1,\ldots,k_n)}$.
 From the arguments above it follows that~$M$ is isomorphic to $\pi^*\left(M_{\mid P'}\right)$.
\end{proof} 

\begin{lem} \label{lem:finDeathBlock}
  Let $M'$ be a pfd persistence module over a product $P'=T'_1\times \ldots\times T'_n$ of finite totally ordered sets. Assume further that $M'$ is $k$-middle exact for all $2\leq k \leq n$ but not $n$-left exact. Then $M'$ has a death block summand.
\end{lem}
\begin{proof}
  The support of $N'=\cap_{i=1}^n \KerDir {M'} i$ is non-empty by assumption, thus has a maximal element $p\in P'$. Consider a non-zero element $m \in \cap_{i=1}^n \KerDir {M'} i (p)$.
  We claim that $m$ has a lift to the minimal point $q=\min(P')$. To prove this, the methodology is the same as in the proof of \Cref{lem:KernelLifting,lem:extension}, solving $k$-lifting problems for increasing values of $2\leq k\leq n$ in $k$-cubes adjacent to $p$ in $P'$.
  Then, the persistence submodule generated by the lift of $m$ at $q$ is isomorphic to a death block. \Cref{lem:injective-summand} ensures that it is a summand of $M'$.
\end{proof}

\begin{proof}[Proof of \Cref{prop:not-n-exact}]
  Assume that $M$ is not $n$-left exact. Since the persistence module $N= \bigcap_{i=1}^n \KerDir M i$ is not $n$-left exact, there must exist $a\in P$ such that its restriction $N_{|a^+}$ is not $n$-left exact either. Then, \Cref{lem:Kersurjective,lem:KernelLifting,lem:discretisation} ensure that there exist a finite subposet $P'\subset a^+$ and a poset morphism $\pi : a^+\to P'$ such that $N_{|a^+}$ is isomorphic to $\pi^*\left(N_{\mid P'}\right)$. By \Cref{lem:finDeathBlock}, the restriction $N_{\mid P'}$ has a death block summand, hence so has $N_{|a^+}$.
  Finally, we extend the death block summand of $N_{|a^+}$ into a death block summand of $N$ by applying \Cref{lem:extension} iteratively for each of the $n$ axes. This death block submodule of $M$ is a summand by \Cref{lem:injective-summand}. The proof of the result when $M$ is not $n$-right exact follows from the previous case by pointwise duality. 
\end{proof}

\section{Case 2: $M$ is $l$-exact but not $(l-1)$-exact.} \label{sec:properlyBoundedExactness}
In this section, we prove \Cref{thm:blockDecomp-exactness} in the particular case of persistence modules which are $l$-exact but not $(l-1)$-exact for some $3\leq l \leq n$.
\begin{prop} \label{pro:Case2}
  Let $M$ be a pfd persistence module over $P$ which is $k$-middle exact for all $2\leq k\leq n$. Assume further that $M$ is $l$-exact but not $(l-1)$-exact for some $3\leq l \leq n$. Then $M$ has an induced block summand.
\end{prop}
\begin{proof}
  Suppose that $M$ is $l$-exact but not $(l-1)$-left exact. Then, there exists a $(l-1)$-slice $W$ of $P$ such that $M_{|W}$ is not $(l-1)$-left exact. By symmetry, assume that $W = \prod_{j = 1}^{l-1} T_{j} \times \prod_{j=l}^n \lbrace t_j \rbrace$. 
  By \Cref{prop:not-n-exact}, the restriction $M_{|W}$ has a death block summand:
  \[ \field_{\prod_{j = 1}^{l-1} J_{j} \times \prod_{j=l}^n \lbrace t_j \rbrace } \hookrightarrow M_{|W},
  \]
  where $J_{j} \subseteq T_{j}$ are non-empty downward-closed intervals. 
  Since $M$ is $l$-left exact, this block extends monomorphically in positive directions of the axes $T_{l},\ldots, T_{n}$.
  Then, applying \Cref{lem:extension} iteratively on these $n-l+1$ axes we get a persistence submodule:
  \[
    \field_{\prod_{j=1}^{l-1} J_{j} \times \prod_{j=l}^n t_j^-} \hookrightarrow M, 
  \]
  which is a direct summand of $M$ by \Cref{lem:injective-summand}. The result when $M$ is $l$-exact but not $(l-1)$-right exact follows from the previous case and pointwise duality.
\end{proof}

\section{Case 3: $M$ is $2$-exact}\label{sec:2-exact}
In this section, we prove \Cref{thm:blockDecomp-exactness} in the particular case of $2$-exact persistence modules. Recall that by \Cref{lem:exactInduction}, $2$-exact persistence modules are $k$-exact for all~$2\leq k\leq n$.
\begin{prop}\label{prop:2-exact}
  Let $M$ be a pfd persistence module over $P$ which is $2$-exact. Then~$M$ has an induced block summand.
\end{prop}
We first prove the result when $P$ is an $n$-cube in \Cref{sec:cube} using restrictions to claws. Then we prove the general result in \Cref{sec:proof-prop-2-exact}.

\subsection{The case of an $n$-cube}\label{sec:cube}
In this section we prove \Cref{prop:2-exact} when $P$ is an $n$-cube using restrictions to claws. Recall our notations for persistence modules on $n$-cubes from \Cref{sec:3d-middle-exactness}.

\begin{lem} \label{lem:clawExtension}
  Let $M'$ be a $2$-exact pfd persistence module over an $n$-cube $C\subset P$. Denote by $\iota \colon L \hookrightarrow C$ the embedding of the claw at $\min(C)$.
  Then $M' \cong \Lan{\iota} \left(\iota^*M'\right)$.  
\end{lem}
\begin{proof}
   Note that one can complete $L$ into $C$ by successively completing all $k$-cubes for all $2\leq k\leq n$ in increasing order starting by cubes adjacent to $\min(C)$. 
   This gives rise to a linear filtration $j_s \colon L_s \hookrightarrow L_{s+1}$, with $L_0 = L$ and $L_N = C$ where in each step exactly one point is added. 
   Denote by $\iota_s \colon L_s \hookrightarrow C$ the inclusions.
   By the composition theorem for Kan extensions \cite[Proposition 3.7.4]{BorceuxI} it is enough to show that $N_s := \Lan{j_s} \left( \iota_s^* M' \right)$ is isomorphic to $\iota_{s+1} ^* M'$ for all $0\leq s \leq N-1$.

   First, consider the case where $j_s\colon L_s \hookrightarrow L_{s+1}$ is the completion of a~$2$-cube:
   \begin{equation*}
    \begin{tikzcd}[ampersand replacement=\&]
      \bullet \&  \\ 
      \bullet \arrow[u] \arrow[r] \& \bullet
      \end{tikzcd}
      \hookrightarrow
      \begin{tikzcd}[ampersand replacement=\&]
        \bullet \arrow[r]\& \bullet  \\ 
        \bullet \arrow[u] \arrow[r] \& \bullet\arrow[u]
      \end{tikzcd}
   \end{equation*}
   Then, the restrictions of $N_s$ and $\iota_{s+1}^*M'$ to~$L_s$ are isomorphic, since they are both isomorphic to $\iota_{s}^*M'$. Moreover, at the $2$-cube $C_s$ having the unique additional point $q\in L_{s+1} \setminus L_s$ as top-right corner, the left Kan extension $N_s = \Lan{j_s} \left( \iota_{s}^*M' \right)$ is a pushout; see {\cite[IX.3, Theorem 1]{Maclane}}. Meanwhile, the $2$-right exactness of $M'$ and the right-exactness of the restriction functor ensure that the restriction of $\iota_{s+1}^* M'$ to the $2$-cube $C_s$ is also a pushout. Uniqueness of pushouts thus yields $N_s \cong \iota_{s+1}^* M'$. 
   
   In the case where $j_s\colon L_s \hookrightarrow L_{s+1}$ is the completion of a $k$-cube without its apex to a complete $k$-cube, we observe that there are $k$ new $(k-1)$-faces added to the graph. We can illustrate the situation when $k=3$:
   
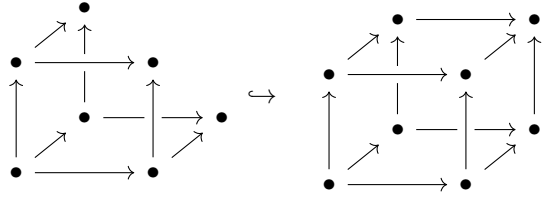
\begin{equation*}
    \begin{tikzcd}[ampersand replacement=\&, row sep=small, column sep=small]
      \& \bullet\arrow[from=dd] \& \& \\
      \bullet\arrow[ur] \arrow[rr, crossing over] \& \&\bullet \\ 
      \& \bullet\arrow[rr] \& \& \bullet\\
      \bullet \arrow[uu]  \arrow[ur] \arrow[rr] \& \& \bullet \arrow[ur] \arrow[uu, crossing over] \& \\
    \end{tikzcd}
    \hookrightarrow\quad
    \begin{tikzcd}[ampersand replacement=\&, row sep=small, column sep=small]
      \& \bullet\arrow[from=dd] \arrow[rr]\& \& \bullet\\
      \bullet\arrow[ur] \arrow[rr, crossing over] \& \&\bullet\arrow[ur] \\ 
      \& \bullet\arrow[rr] \& \& \bullet\arrow[uu]\\
      \bullet \arrow[uu]  \arrow[ur] \arrow[rr] \& \& \bullet \arrow[ur] \arrow[uu, crossing over] \&
    \end{tikzcd}
   \end{equation*}
   By $(k-1)$-right exactness of $M'$, these $(k-1)$-faces are all pushouts and the isomorphism $N_s \cong \iota_{s+1}^* M'$ follows similarly from the uniqueness of pushouts.
\end{proof}

Persistence modules over $L$ are not always interval-decomposable, see for instance \Cref{ex:2-mid_not_3-mid}. However, restrictions of $2$-exact persistence modules do, as shown by the following lemma.

\begin{lem} \label{lem:DecompRestr}
  Consider the setting of \Cref{lem:clawExtension}. Then $\iota^*M'$ is interval-decomposable.
\end{lem}
\begin{proof}
  Since $M'$ is pfd and $C$ is a finite poset, a simple induction shows that $M'$ decomposes as a direct sum of indecomposable modules, so we can assume without loss of generality that $M'$ is indecomposable. In the present case this implies that the restriction $\iota ^* M'$ is also indecomposable. Indeed, we have the following isomorphisms:
  \begin{equation}
    \label{eq:Lan_restr_adj_bis}
    \begin{split}
      \Hom_{\Mod \field L} \left( \iota^* M' ,\iota^* M'  \right) &\cong  \Hom_{\Mod \field P} \left( \Lan {\iota} \left( \iota^* M' \right) , M' \right) \\
      &\cong \Hom_{\Mod \field P} \left( M', M' \right),
    \end{split}
  \end{equation}
  where the first isomorphism follows from~\eqref{eq:Lan_restr_adj} and the second from~\Cref{lem:clawExtension}. These are even isomorphisms of $\field$-algebras, as the involved isomorphisms of vector spaces are natural. Having $M'$ indecomposable implies that its endomorphism ring contains no non-trivial idempotents, which in turn implies by~\eqref{eq:Lan_restr_adj_bis} that the endomorphism ring of $\iota^* M'$ also has no non-trivial idempotents, and therefore that $\iota^* M'$ itself is indecomposable.

  Let us now show that $\iota^* M'$ is an interval module.
  We prove the assertion by induction on $n\geq 2$. 
  When $n=2$, the persistence module $\iota^* M'$ must be an interval module by the decomposition theory of posets of type $A_n$.
  Assume $n\geq 3$. 
  If the support of $\iota^* M'$ does not intersect the interior of all arms of $L$, then it is an interval module by the induction hypothesis.
  So assume that the support of $\iota^*M'$ intersects the interior of all arms of $L$. 
  All structure maps of $\iota^*M'$ must be surjective, as we could otherwise split off a summand on an arm. 
  
  Assume first that one of the structure maps of $\iota^* M'$ is not injective on an arm of $L$. Then, there must be a non-trivial element $m \in M'(\min L)$ which is mapped to zero along, say, the $i$-th axis.
  By $2$-left exactness of $M'$, the structure maps along the other axes must be injective. Thus, the submodule of $M'$ spanned by $m$ is isomorphic to $\field_I$, where $I \subseteq C$ is the $(n-1)$-slice parallel to the $i$-th axis containing $\min(C)=\min(L)$. The interval $I$ is a death block, therefore $\field_I$ is a summand of $M'$, by \Cref{lem:injective-summand}. By our assumption that $M'$ is indecomposable, this implies that $M' \cong \field_I$, hence the restriction $\iota^* M'$ must be an interval module. 
  
  Assume now that the structure maps of $\iota^* M'$ are injective on all arms of $L$. Then, all structure maps of $\iota^* M'$ are isomorphisms. Since $\iota^* M'$ is indecomposable, then so is $M'(\min L)$, i.e., it is a one-dimensional vector space. Thus $\iota^* M'$ is isomorphic to $\field_L$, an interval module.
\end{proof}

\begin{prop} \label{pro:discrete2ex}
  Consider the setting of \Cref{lem:clawExtension}. Then $M'$ is block-decomposable.
\end{prop}

\begin{proof}
  By \Cref{lem:clawExtension}, we know that $M' \cong \Lan{\iota} \left( \iota^*M' \right)$. By \Cref{lem:DecompRestr}, there is a direct sum decomposition $\iota^* M' \cong \bigoplus _{I \in \II} \field_I$ for a set of intervals $\II$ of $L$.
  Then, by \Cref{lem:clawExtension} and additivity of left Kan extensions, we obtain:
  \[
    M' \cong \Lan{\iota} \left(\iota^* M'\right) \cong \Lan {\iota} \left(\bigoplus_{I \in \II} \field_I\right) \cong \bigoplus_{I \in \II} \left(\Lan{\iota}\field_I\right)
  \]
  It is a simple exercise to check that for any interval $I\in\II$, the persistence module $\Lan{\iota}\field_I$ is an induced block module.
\end{proof}

\begin{rk}\label{rk:discrete2ex}
  In fact, the proof of \Cref{pro:discrete2ex} ensures that the block summands of~$M'$ are obtained by left Kan extension from the interval summands of the restriction of~$M'$ to the claw.
\end{rk}

\subsection{Proof of \Cref{prop:2-exact}}\label{sec:proof-prop-2-exact}
Let $M$ be a pfd persistence module over $P$. Recall from \Cref{sec:study-kernel-submodules} that~$\KerDir{M}{i}$ is the persistence submodule of~$M$ consisting of all elements which are eventually mapped to zero by the structure maps of~$M$ along the~$T_i$-axis. Similarly, we denote by~$\ImDir{M}{i}$ the persistence submodule of $M$ which consists of elements that have preimages by all structure maps of $M$ parallel to the $T_i$-axis. The proof of \Cref{prop:2-exact} is divided into two subcases, dealt with in \Cref{lem:2-exact-1,lem:2-exact-2} respectively.

\begin{lem}\label{lem:2-exact-1}
  Let $M$ be a pfd persistence module over $P$ which is $2$-exact. Assume further that $\KerDir{M}{i} \cap \ImDir{M}{i} \neq 0$ for some $i\in\{1,\ldots,n\}$.
  Then $M$ has an induced block summand.
\end{lem}
\begin{proof}
    If $N:= \KerDir{M}{i} \cap \ImDir{M}{i} \neq 0$ for some $i\in\{1,\ldots,n\}$, (say $i=1$ the other cases being symmetric), then there is another index $j\in\{1,\ldots,n\}\setminus\{i\}$ (say $j=2$ the other cases being symmetric) and elements $t_k \in T_k$ for $k \geq 3$ such that the restriction of $N$ to the $2$-slice $S = T_1 \times T_2 \times \prod_{k\geq 3}  \lbrace t_k \rbrace$ is non-zero. By the $2$-parameter block-decomposition result (\Cref{thm:2d-block-rect-decomp}), then the restriction of $M$ to $S$ has a death block summand isomorphic to $\field_B$ where $B = J_1\times T_2 \times \prod_{k\geq 3}  \lbrace t_k \rbrace$ for some downward-closed interval $J_1\subset T_1$. This yields a monomorphism $\field_B \hookrightarrow M_S$. By \Cref{lem:extension}, this can be extended in decreasing $T_3$-direction to get a monomorphism:
    \begin{equation*}
      \field_{J_1\times T_2 \times t_3^- \times \prod_{k\geq 4}  \lbrace t_k \rbrace} \hookrightarrow M_{T_1 \times T_2 \times t_3^- \times \prod_{k\geq 4}  \lbrace t_k \rbrace}.
    \end{equation*}
    By $2$-left exactness of $M$, one can extend this submodule injectively in increasing $T_3$-direction to get a monomorphism:
    \begin{equation*}
      \field_{J_1\times T_2 \times T_3 \times \prod_{k\geq 4}  \lbrace t_k \rbrace} \hookrightarrow M_{T_1 \times T_2 \times T_3 \times \prod_{k\geq 4}  \lbrace t_k \rbrace}.
    \end{equation*}
    Proceeding similarly on each remaining axis $T_k$ for $4\leq k \leq n$, we get a submodule:
    \begin{equation*}
      \field_{J_1\times T_2\times \ldots\times T_k} \hookrightarrow M,
    \end{equation*}
    which is death block submodule and therefore a direct summand by \Cref{lem:injective-summand}.
\end{proof}

\begin{lem}\label{lem:2-exact-2}
  Let $M$ be a pfd persistence module over $P$ which is $2$-exact. Assume further that $\KerDir{M}{i} \cap \ImDir{M}{i} = 0$ for all $i\in\{1,\ldots,n\}$.
  Then $M$ has an induced block summand.
\end{lem}
\begin{proof}
    We proceed by induction on~$n\geq 2$. If~$n=2$, we already know that~$M$ is block-decomposable (\Cref{thm:2d-block-rect-decomp}). That a $2$-exact 2-parameter persistence module must have only induced block summands is an easy exercise. Suppose now~$n\geq 3$.
    
    It follows from~$2$-right exactness that~$M = \ImDir{M}{i} + \ImDir{M}{j}$ for all~$i \neq j$.  First, assume that~$\ImDir{M}{i} \cap \ImDir{M}{j} = 0$ for some~$i  \neq j$. Consider a summand~$N := \ImDir{M}{i}$ which is non-zero. By the assumption of the lemma, all structure maps of~$N$ that are parallel to the~$T_i$ axis are isomorphisms.
    The restriction of~$N$ to any~$(n-1)$-slice parallel to the~$T_i$-axis has an induced block summand by the induction hypothesis, hence~$N$ itself has an induced block summand, and so does~$M$.

    Assume now that $N':=\bigcap_{i=1}^n \ImDir{M}{i} \neq 0$. All structure maps of $N'$ are surjective and all its structure spaces are non-zero, hence this submodule of $M$ has a direct summand isomorphic to $\field_P$ by \cite[Lemma 2.3]{Botnan2018a}. By \Cref{lem:injective-summand}, the submodule~$\field_P$ of $M$ is a direct summand.

    Finally, assume that there exists $2\leq k\leq n-1$ such that all $(k+1)$-fold intersections of the submodules $\ImDir{M}{i}$ for $i\in\{1,\ldots,n\}$ vanish, but that at least one $k$-fold intersection is non-zero. From this property and the block-decomposition result on an $n$-cube (\Cref{pro:discrete2ex}, \Cref{rk:discrete2ex}), it follows that the restriction of $M$ to any $n$-cube $C$ of $P$ has a block decomposition in which only block modules supported on $k$-cubes adjacent to the maximum of $C$ appear. As a consequence:
    \[
      M = \sum_{\substack{\JJ \subseteq \{1,\ldots,n\} \\ |\JJ| = k}} \quad\bigcap_{i \in \JJ} \ImDir{M}{i},
    \]
    and this sum is in fact a direct sum by our assumption on $(k+1)$-folds intersections of the submodules $\ImDir{M}{i}$. Consider then a non-zero summand $N'' =  \bigcap_{i \in \JJ} \ImDir{M}{i}$ for some $\JJ\subset \{1,\ldots,n\}$ with $|\JJ |= k$. The structure maps of $N''$ which are parallel to the $T_j$-axis for $j\in\{1,\ldots,n\}\setminus\JJ$ are all isomorphisms by the assumptions of the lemma. Hence $N''$ is simply the extension by isomorphisms of any one of its restriction to $(n-k)$-slices with constant parameters along the axes $T_j$ for $j\in\{1,\ldots,n\}\setminus\JJ$. But these restrictions to $(n-k)$-slices have a block summand by our induction hypothesis. Hence so does $N''$, and so does $M$.
\end{proof}

\bibliographystyle{plain}
\bibliography{Literature}
\end{document}